\documentclass{article}
\usepackage[utf8]{inputenc}

\usepackage{amsmath}
\usepackage{amsfonts}
\usepackage{amsthm}
\usepackage{mathtools}
\usepackage{biblatex}

\usepackage{geometry}
\usepackage[english]{babel}
\newtheorem{theorem}{Theorem}[section]

\newtheorem{lemma}[theorem]{Lemma}

\title{The Exceptional Selfcondensability\\ of Powers of Five}
\author{O.M. Cain \\ onnomawcain@gmail.com}
\date{October 2019}

\begin{document}

\maketitle

\textbf{Abstract.} We show any power of $5$ may be expressed arithmetically with the digits of its decimal representation. We also show powers of $5$ (in decimal) contain any amount of zeros in a row.

\section{The Condensing Game}
\begin{center}
    ``\textit{... the soul passes out of hypotheses, and goes up to a principle which is above hypotheses, making no use of images as in the former case, but proceeding only in and through the ideas themselves.}" \\
    -Plato, \textit{The Republic, Book VI} \\\ \\
    ``\textit{... the Romance of Mathematics has dissolved before our eyes.}" \\
    -Lakoff \& N\'u\~nez, \textit{Where Mathematics Comes From} \\
\end{center}

It is often great fun determining whether some value can be created arithmetically from some given list of numbers. For instance, can we make $30$ from four $4$'s? Spoiler, yes:
$$\frac{(4+\frac{4}{4})!}{4}=\frac{5!}{4}=5\cdot3\cdot 2\cdot 1=30.$$
Can you solve $31$?

We are going to use the verb \textit{condensing} to mean playing this game of arithmetic rearrangement. So in the previous example we would say four $4$'s have been \textit{condensed} into $30$. The verb  nicely communicates the feeling of boiling down a soup of many numbers into a single result.

A lot of people have analyzed the Condensing Game for four $4$'s so we will trek a different trail\footnote{See, for example, Paul Bourke's solution list at
\texttt{http://paulbourke.net/fun/4444/}}. We are interested in numbers condensed from their own digits such as
$$25 = 5^2\quad\text{or}\quad 2187=(2+1^8)^7.$$
Numbers for which this is possible we call \textit{selfcondensable}\footnote{Other candidates besides \textit{condensing} were \textit{composing} and \textit{assembling} with alternatives \textit{recomposable} and \textit{autosemblable} to \textit{selfcondensable}.}. And to be clear, when we talk about a number's own digits, we mean the digits of its decimal representation (as opposed to binary or hexadecimal). An analysis of selfcondensabilty in other bases would be a wonderful excursion to meander down but we won't take that trail here.

Selfcondensable numbers are very rare for small numbers (say numbers $7$ digits or less). Perfect powers are usually exceptions to this observation but some families of powers don't become selfcondensable until very large. For example, $10^n$ doesn't seem to become selfcondensable until
$$10^{11}=100000000000=((1+0!)((0!+0!+0!)!-0!))^{(0!+0!+0!)!(0!+0!)-0!}.$$
Remember that $0!=1$ for some reason\footnote{See Numberphile: \texttt{https://www.youtube.com/watch?v=Mfk\_L4Nx2ZI}}.

And now we reach a bend from which our summit can be seen. Surprisingly, the powers of five all seem to be selfcondensable:
$$25=5^2,\quad 125=5^{1+2},\quad 625=5^{6-2},\quad 3125=5^{3\cdot 1 + 2},\quad ...$$
Condensing $5^n$ seems to get easier as $n$ gets bigger. Maybe it goes on forever? But we want a proof that all powers of five -- all infinitely many of them -- are selfcondensable. We will create such a proof in sections 5, 6, and 7. We must also play a bit with language in section 3.

Before closing the introduction, there is another trail we won't be taking but which deserves mention. Some numbers can be condensed into the same value in different ways. Try condensing $2,3,$ and $5$ for a minute or so seeing what values one can create. Find $13$? There are a few routes to it:
$$2^3+5=3\cdot 5 - 2=2\cdot 5 + 3=13.$$
It would be interesting to determine which target values have the most distinct condensings (not counting the order of addition or multiplication or things of that kind as distinct). Are there any $3$ numbers which can be condensed into a single target value in $4$ or more distinct ways? Or is $13$ a record-holder for $3$ starting values?

\section{Da Rules}
\begin{center}
    ``\textit{True life knows and rules itself}" \\
    -George MacDonald, \textit{Unspoken Sermons}
\end{center}
We should probably nail down what exactly is allowed in the Condensing Game before going any further. We will use six rules.
\begin{enumerate}
    \item Addition (ex. $3+15=18$)
    \item Subtraction (ex. $3-15=-12$)
    \item Multiplication (ex. $3\cdot15=45$)
    \item Division (ex. $\frac{15}{3}=5$)
    \item Powers (ex. $15^3=3375$)
    \item Factorial (ex. $3!=6$)
\end{enumerate}

The Condensing Game is usually played with more rules than these (logarithms, square roots, concatenation, decimals, ...). For example, condensing five $3$'s into $37$ has two solutions using concatenation\footnote{See \textit{Moscow Puzzles} [1] problem 77B},
$$37\quad=\quad33+3+\frac{3}{3}\quad=\quad\frac{333}{3\cdot3},$$
and a third with decimals,
$$37=3\cdot3\cdot3+\frac{3}{.3}.$$
There is a solution with only our six rules (in fact, it uses only five). Can you find it? Is there a solution with just four of our rules?

The Condensing game is usually too easy -- and therefore boring -- with a lot of rules. For example, can you condense $100$ from six identical digits? If we have concatenation, this is easy since
$$\frac{777-77}{7}=100$$
holds true for any digit\footnote{See \textit{Moscow Puzzles} [1] problem 77C} (not just $7$). There is an easy solution if we have decimals as well,
$$\Big(\frac{7}{.7}\Big)^{\frac{7}{7}+\frac{7}{7}}=100.$$
But what about our rules? Now the problem is interesting since
$$(2+2+2+2+2)^2\quad=\quad4!\cdot4+4+\frac{4-4}{4}\quad=\quad5!-5\cdot(5-(5-5)!^5)$$
$$=\quad7\cdot(7+7)+\frac{7+7}{7}\quad=\quad9\cdot9+9+9+\frac{9}{9}\quad=\quad100$$
appear the only solutions. Is there a solution for $1,3,6,$ or $8$?

There is another case in which the game is too easy\footnote{See Numberphile: \texttt{https://www.youtube.com/watch?v=Noo4lN-vSvw}}. With just square roots, logarithms, and division we can condense any value from four $4$'s:
$$n=\log_{\sqrt{4}/4}\log_4\underbrace{\sqrt{\sqrt{\cdots\sqrt4}}}_{n}.$$
The trick works because we can nest as many square roots as needed. Similarly, if we have instead the natural logarithm we can condense any value from any four equal numbers (say, $m,m,m$ and $m$):
$$n=\ln\left[\frac{\ln\underbrace{\sqrt{...\sqrt{m}}}_n}{\ln m}\right]/\ln\left[\frac{\ln\sqrt{m}}{\ln m}\right].$$

The six rules above are chosen because they strong enough to get the job done but not too strong that the game becomes boring. In fact, they may make things too easy. In Section 8 we will try to do without rules \#5 and \#6.

\section{Using ``Multiset" in a Sentence}
\begin{center}
    ``\textit{You see, this quadrinomial has been improperly factored. He forgot to double cube root the bottom partial nominator.}" \\
    -Betty Grof, \textit{Adventure Time}
\end{center}

A large vocabulary -- like all good things -- can be misused. Familiarity with a lot of words lets one speak clearly but can also hatch into that horrible dragon: \textit{Jargon}. To continue our conversation about condensability we should add a word to our mathematical vocabulary. And we do so with trembling in fear do not understand what we say.

The word \textit{multiset} is roughly self-explanatory. It means a set of numbers where each element can appear multiple times instead of just once or none at all. For example, $\{7,2,6\}$ is a set but $\{2,6,2,7,7,2\}$ is a \textit{multiset}. We say that $2$ has a \textit{multiplicity} of $3$ in this particular multiset since it appears $3$ times.

We should know also how to add and subtract multisets. Here is an example of each:
$$\{7,6,2\}+\{3,7,7,2\}=\{2,2,3,6,7,7,7\}$$
$$\{2,6,2,7,7,2\}-\{7,6,2\}=\{2,2,7\}$$
In English, adding or subtracting multisets means adding or subtracting the multiplicities of their elements. These examples are something like how one might use mutlisets in a sentence.

We should note when subtracting, one might create \textit{negative multiplicities}. We simply leave such an oddity undefined and avoid it entirely (like dividing by zero).

Lastly, for a multiset $S$, we will use the symbol ``$[S]$" to mean the \textit{set} of elements of $S$. For example, if $S=\{2,2,3,6,7,7,7\}$ then $[S]=\{2,3,6,7\}$.

\section{Made-up Words}
\begin{center}
    ``\textit{A sovereign shame so elbows him}" \\
    -Shakespeare, \textit{King Lear}\\\ \\
    ``\textit{Shmow-zow!}" \\
    -Finn Mertens, \textit{Adventure Time}
\end{center}

It is not always enough to add existing words to one's vocabulary since when a totally new concept is to be communicated, there is no language at first to describe it. So for any significantly novel idea one must create language. This is done in one or both of two ways. 1) Entirely new words are invented. Shakespeare, for example, is considered to have invented the word ``elbow". And remember every word was at some point spoken for the first time. Or 2) we can give special meanings to existing words. 

This second case is exactly what we experienced in high school math with the words ``divisor" and ``imaginary". Similarly, In group theory (a particular branch of math), the words ``orbit" and ``generator" are of this sort. Interestingly, there are some words which gain a special mathematical meaning which outlives the word's original conversational meaning. That is what happened with ``matrix". The word originally meant a place of development or origin. It came from the French \textit{matrice} meaning a pregnant animal (for example one 19th century author calls the human mind a ``thought matrix"\footnote{``\textit{he speaks ... to the thought-matrix}" \\-Unspoken Sermons II (1885)}). The same is true of the word ``quaternion"\footnote{``\textit{the mail-coach was moving away with a fresh quaternion of horses}" \\-Robert Falconer (1868)}. It appears surprisingly common for the conversational meaning of a word to die out while the mathematical meaning survives.

The language of mathematicians is a strange blend of English, symbols, made-up words, and code-words of special meaning (``Mathlish" is a good name for the language). But our job here is not philology. We had better return to our present trail. In the first section we gave special meaning to the word ``condensing". Here we will make up ``words" entirely. Or -- to be very technically accurate -- we are making up \textit{mathematical graphemes}.

Given a multiset $S$, we will use the symbol ``$V(S)$" to mean all the values which can be condensed from $S$. For example:
$$V(\{2,5\})=\{2+5,\ 2-5,\ 5-2,\ 2\cdot 5,\ \frac{2}{5},\ \frac{5}{2},\ 2^5,\ 5^2,\ 5!,\ (2+5)!,\ ...\}$$
$$=\{7,\ -3,\ 3,\ 10,\ 0.4,\ 2.5,\ 32,\ 25,\ 120,\ 5040,\ ...\}$$
If any reader craves a rigorous definition we will satiate them with: \\
``$V(S)$ is the smallest set such that 
\begin{enumerate}
    \item If $|S|=1$ then $S\subseteq V(S)$.
    \item If $|S|\ge 2$ then for all multisets $A$ and $B$ such that $A+B=S$
    $$\bigcup_{a\in V(A)}\bigcup_{b\in V(B)}\{a+b,\ a-b,\ b-a,\ ab,\ \frac{a}{b},\ \frac{b}{a},\ a^b,\ b^a\} \subseteq V(S)$$
    \item If $a\in V(S)$ then $a!\in V(S)$."
\end{enumerate}

So the sentence ``$237$ is not selfcondensable" we can now translate equivalently as ``$237\not\in V(\{2, 3, 7\})$". 

Next we will give meaning to the grapheme ``$E_k$". It will refer to all numbers which can be condensed from any $k$ decimal digits. Translated into an equation, this becomes
$$E_k=\bigcap _{\substack{|S|=k\\ [S]\subseteq D}} V(S)$$ 
where $D=\{0,1,2,...,9\}$ is simply the set of decimal digits. For example, we will find out later that any $5$ decimal digits can be condensed into the number $2$. This fact can be stated concisely with ``$2\in E_5$". 

We should note an important fact about $E_k$.

\begin{lemma}
 $E_k\subseteq E_{k+1}$ for $k\ge 1$. Or in English, if a number can be condensed from any $k$ decimal digits, we can be sure it is also condensable from any $k+1$ decimal digits.
\end{lemma}
\begin{proof}
Suppose $n$ can be condensed from any $k$ decimal digits. We must show that given any multiset $S$ of $k+1$ such digits, we can still condense $n$ therefrom. But this is easy! Simply pick two entries in $S$, say $a$ and $b$. If they are not equal suppose $a$ is the larger (i.e. $a\ge b$). We simply subtract $b$ from $a$ making a new multiset, $S'=S-\{a,b\}+\{a-b\}$, having $k$ elements with the property that $V(S')\subseteq V(S).$ But we know from our assumption that $n\in V(S')$.
\end{proof}

Lastly, we define ``$\delta(n)$" to be the smallest number of digits such that we are assured $n$ is condensable therefrom. Rigorously
$$\delta(n)=\min\{k:n\in E_k\}.$$
It will also be useful here to give a formula for the number of digits in an integer's decimal expansion:

\begin{lemma}
    A whole number $n$ has $\lfloor \log_{10}n\rfloor+1$ digits in its decimal representation.
\end{lemma}
\begin{proof}
    Remember ``$\lfloor\ \rfloor$" is Mathlish for ``round down". Say $n$ has $k$ digits. Then $k$ is the largest integer such that $n\ge 10^{k-1}$. By taking logarithms we see
    $$\log_{10} n\ge =log_{10}(10^{k-1})=k-1$$
    and since $k$ is the \textit{largest integer} with this property, we are assured that the ``$\ge$" will become an ``$=$" if we round down $\log_{10}n$. In other words,
    $$\lfloor \log_{10}n\rfloor=k-1.$$
    Or equivalently, $k=\lfloor \log_{10}n\rfloor+1.$ 
\end{proof}

This is all to say if $\delta(n)\leq \lfloor \log_{10}n\rfloor+1$ then $n$ must be selfcondensable.

\section{Crimps}
\begin{center}
    ``\textit{Crimp: a hold which is only just big enough to be grasped with the tips of the fingers.}" \\
    -Glossary of Climbing Terms, \textit{Wikipedia}
\end{center}

We have now a view of the summit and some good tools. It is time to climb. We first mark out some grab-holds -- or ``lemmas" as they are called. To make things go faster, we will assume that the digits we work with are non-zero. If any zeros are encountered in the wild, they can simply be made non-zero by factorial since $0!=1$.

\begin{lemma}
$\delta(1)\le 4$. Or in other words, any $4$ decimal digits can be condensed into $1$.
\end{lemma}
\begin{proof}
Suppose we have a multiset, $S=\{a,b,c,d\}$, of four elements. If any element, say $a$, is $1$ we are done since $a^{b+c+d}=1^{b+c+d}=1\in V(S)$. Similarly, if any two elements are equal, say $a=b$, we are done since $\Big(\frac{a}{b}\Big)^{c+d}=1^{c+d}=1\in V(S)$. The same goes if two elements are $1$ apart, say if $a-b=1$.

The only case left is when $a,b,c,$ and $d$ are at least $2$ apart from each other. The only possibilities are
$$\{2,4,6,8\},\quad \{2,4,6,9\},\quad \{2,4,7,9\},\quad \{2,5,7,9\},\quad \text{and} \quad \{3,5,7,9\}$$
which we can take care of with our bare hands:
$$\Big(\frac{2+4}{6}\Big)^8=\Big(\frac{2+4}{6}\Big)^9=(2\cdot4-7)^9=(2\cdot5-9)^7=(9-3-5)^7=1.$$
\end{proof}

\begin{lemma}
If $1\not\in V(\{a,b,c\})$ for $0\le a,b,c\le 9$ then $\{a,b,c\}$ is one of $\{4,6,8\}, \{4,7,9\},$ or $\{5,7,9\}$.
\end{lemma}
\begin{proof}
As observed in the previous lemma, if any of $a,b,$ or $c$ is $1$, if any of $a,b,$ and $c$ are equal, or if any of $a,b,$ and $c$ are $1$ apart, then we are done. If not, then suppose $a,b,$ and $c$ are sorted so $a<b<c$. If $a+b$ or $ab$ is equal or $1$ apart from $c$ then we are done since one of $\frac{a+b}{c}, \frac{ab}{c}, a+b-c, c-a-b, c-ab, $ or $ab-c$ is equal to $1$. That leaves only $\{4,6,8\}, \{4,7,9\},$ and $\{5,7,9\}$.
\end{proof}

We could have proved Lemma 5.1 with 5.2 if we wanted. But the author chose not to simply because the resulting proof ``felt uglier" than given 5.1 a standalone proof.

\begin{lemma}
$\delta(2)\le 5$.
\end{lemma}
\begin{proof}
Suppose $S$ a multiset with $5$ elements. If $S$ contains $0,1,2,$ or $3$ we are done since, by the previous lemma, we can condense the other $4$ elements into $1$ and observe in each case
$$0!+1=2,\quad 1+1=2,\quad 2\cdot 1=2,\quad\text{and}\quad 3-1=2.$$

Next if any two elements in $S$ are equal, say $a=b$, we may show $2\in V(S).$ If the remaining $3$ elements can be condensed into $1$ we are done since $\frac{a}{b}+1=2$. If not, they must be one of $\{4,6,8\}, \{4,7,9\},$ or $\{5,7,9\}$ by the previous lemma. But in these cases we condense
$$(4+6-8)\frac{a}{b}=\frac{4a}{(9-7)b}=(7-5)\Big(\frac{a}{b}\Big)^9=2.$$
The same trick can be pulled if $a-b=1$.

That leaves the case in which we have $5$ elements $\ge 4$ all at least $2$ apart. But this is impossible! The fact that $0,1,2,3\not\in S$ tells us
$$\max S-\min S\le 5$$
whereas having all elements $2$ apart tells us
$$\max S-\min S\ge 8,$$
a contradiction. So it turns out we are done.
\end{proof}

\begin{lemma}
$\delta(3),\delta(6)\le 6.$
\end{lemma}
\begin{proof}
We only need prove the lemma for $\delta(3)$ since $6=3!$. Similarly to the previous lemma, note that if $S$ contains any of $0,1,2,3,4,5,6$ we are done. The other $5$ elements can be condensed into $1$ or $2$ by the previous lemmas. In each case we have
$$0!+\mathbf{2}=1+\mathbf{2}=2+\mathbf{1}=3\cdot\mathbf{1}=4-\mathbf{1}=5-\mathbf{2}=\frac{6}{\mathbf{2}}=3.$$
The $1$'s and $2$'s made by condensing are bold for clarity.

That leaves the case $[S]\subseteq\{7,8,9\}$. Let's assign names to the elements of $S$. Say $S=\{a,b,c,d,e,f\}$ with the elements in sorted (=non-decreasing) order. We make a second multiset of $3$ elements, $S'=\{b-a, d-c, f-e\}$, with the nice property that $V(S')\subseteq V(S)$. The only possibilities for $S'$ are 
$$\{0,0,0\},\quad \{0,0,1\},\quad \{0,0,2\},\quad \text{and}\quad\{0,1,1\}.$$
In all cases $3\in V(S')$ since
$$0!+0!+0!=0!+0!+1=0!\cdot 0! + 2=0!+1+1=3$$
\end{proof}

\begin{lemma}
$\delta(4),\delta(5)\le 7$.
\end{lemma}
\begin{proof}
Suppose $|S|=7$. If $S$ contains any of $0,1,...,8$ then $4,5\in V(S)$ since we may condense the other $6$ elements into $1,2,3,$ or $6$ creating
$$0!+\mathbf3=1+\mathbf3=2+\mathbf2=3+\mathbf1=4\cdot\mathbf1=5-\mathbf1=6-\mathbf2=7-\mathbf3=\frac{8}{\mathbf2}=4,$$
$$\mathbf6-0!=\mathbf6-1=2+\mathbf3=3+\mathbf2=4+\mathbf1=5\cdot \mathbf1=6-\mathbf1=7-\mathbf2=8-\mathbf3=5.$$

That leaves us the case $S=\{9,9,9,9,9,9,9\}$ in which case we condense
$$\frac{9+9+9+9}{9}+9-9=4\quad\text{and}\quad \frac{9\cdot9-9-9-9-9}{9}=5.$$
\end{proof}

We should package up these lemmas into a single table.
\begin{center}
\begin{tabular}{ c | c c c c c c}
 $n$ & 1 & 2 & 3 & 4 & 5 & 6 \\ \hline
 $\delta(n)\le$ & 4 & 5 & 6 & 7 & 7 & 6    
\end{tabular}
\end{center}

\section{Belaying}
\begin{center}
    ``\textit{while it had established a rule and order, the chief aim of that order was to give room for good things to run wild.}" \\
    -G.K. Chesterton, \textit{Orthodoxy}\\\ \\
    ``\textit{Belay: something (such as a projection of rock) to which a person or rope is anchored.}" \\
    -\textit{Merriam-Webster}
\end{center}

It would be easy to go on making lemmas for $\delta(7), \delta(8), \delta(9), ...$ and so on. But this is tedious and we have infinitely many numbers to go! Instead we make a single lemma that lets us easily lock down a bound for any $\delta(n)$. The bound won't be the best possible but it will get us up the mountain.

\begin{lemma}
$\delta(a+b), \delta(ab)\le \delta(a)+\delta(b)$ 
\end{lemma}
\begin{proof}
The symbols above become self-explanatory when we take a particular example. For example, we can condense the number $6$ from any six digits and the number $5$ from any seven digits. That means we can condense $11$ or $30$ from any thirteen digits. For $30$, condense a $5$ from seven of the thirteen digits. Use the remaining six digits to condense the number $6$. Then simply multiply $5\cdot6=30$. In short, we say $\delta(11)=\delta(5+6)\le \delta(5)+\delta(6)=7+6=13.$

In general, if we can condense $a$ from any $k$ digits and $b$ from any $m$ digits, then we can condense $a+b$ and $ab$ from $k+m$ digits. 
\end{proof}

Sometimes there are multiple ways to apply the lemma,
$$\delta(8)=\delta(2\cdot4)\le \delta(2)+\delta(4)=5+7=12$$
$$\delta(8)=\delta(2+6)\le \delta(2)+\delta(6)=5+6=11,$$
in which case we take the best result: $\delta(8)\le 11$.

We can now make a much larger table of $\delta(n)$ inching our way up:
\begin{center}
    \begin{tabular}{c|c}
        $n$ & $\delta(n)\le$ \\ \hline
        $1$ & $4$ \\
        $2$ & $5$ \\
        $3$ & $6$ \\
        $4$ & $7$ \\
        $5$ & $7$ \\
        $6$ & $6$ \\
        $7$ & $10$ \\
        $8$ & $11$ \\
        $9$ & $12$ \\
        $10$ & $12$ \\
    \end{tabular}\quad\quad
    \begin{tabular}{c|c}
        $n$ & $\delta(n)\le$ \\ \hline
        $11$ & $13$ \\
        $12$ & $11$ \\
        $13$ & $15$ \\
        $14$ & $15$ \\
        $15$ & $13$ \\
        $16$ & $14$ \\
        $17$ & $18$ \\
        $18$ & $12$ \\
        $19$ & $16$ \\
        $20$ & $14$ \\
    \end{tabular}\quad\quad
    \begin{tabular}{c|c}
        $n$ & $\delta(n)\le$ \\ \hline
        $21$ & $16$ \\
        $22$ & $18$ \\
        $23$ & $19$ \\
        $24$ & $13$ \\
        $25$ & $14$ \\
        $26$ & $18$ \\
        $27$ & $18$ \\
        $28$ & $17$ \\
        $29$ & $20$ \\
        $30$ & $13$ \\
    \end{tabular}
\end{center}
\begin{center}
    \begin{tabular}{c|c}
        $n$ & $\delta(n)\le$ \\ \hline
        $31$ & $17$ \\
        $32$ & $18$ \\
        $33$ & $19$ \\
        $34$ & $20$ \\
        $35$ & $17$ \\
        $36$ & $12$ \\
        $37$ & $16$ \\
        $38$ & $17$ \\
        $39$ & $18$ \\
        $40$ & $18$ \\
    \end{tabular}\quad\quad
    \begin{tabular}{c|c}
        $n$ & $\delta(n)\le$ \\ \hline
        $41$ & $19$ \\
        $42$ & $16$ \\
        $43$ & $20$ \\
        $44$ & $20$ \\
        $45$ & $19$ \\
        $46$ & $23$ \\
        $47$ & $23$ \\
        $48$ & $17$ \\
        $49$ & $20$ \\
        $50$ & $19$ \\
    \end{tabular}\quad\quad
    \begin{tabular}{c|c}
        $n$ & $\delta(n)\le$ \\ \hline
        $51$ & $23$ \\
        $52$ & $22$ \\
        $53$ & $24$ \\
        $54$ & $18$ \\
        $55$ & $20$ \\
        $56$ & $21$ \\
        $57$ & $22$ \\
        $58$ & $25$ \\
        $59$ & $25$ \\
        $60$ & $18$ \\
    \end{tabular}
\end{center}

This still isn't enough. To reach the summit, we need a bound for every $\delta(n)$. And we can create one! But to do so first requires an introduction to base-$6$ numbers. 

When we write number ``$1729$" we use a sort of short hand for saying 
$$1729 = 1000\cdot\mathbf1 +100\cdot\mathbf7 +10\cdot\mathbf2 +\mathbf9$$
$$ = 10^3\cdot\mathbf1 +10^2\cdot\mathbf7 +10^1\cdot\mathbf2 +\mathbf9$$
But of course, this can be done with any number besides $10$ as our base. When we use $6$ as the base we get
$$1729=6^4\cdot \mathbf1 + 6^3\cdot \mathbf2+6^2\cdot\mathbf0+6^1\cdot\mathbf0+\mathbf1$$
$$=1296\cdot 1+216\cdot 2 + 1$$
Mathematicians sometimes like to write this as $``1729=12001|_6"$.

Now we can go on to

\begin{theorem}
$\delta(n)\le 13\log_6 n + 7$
\end{theorem}
\begin{proof}
Suppose $a_k...a_0|_6$ is the base-$6$ representation of $n$. Or in other words, suppose $n=a_0+6(a_1+6(a_2+6(a_3+...))).$ Since this expression is just a bunch of additions and multiplications, we can pull it apart with the previous lemma and get
$$\delta(n)\le k\delta(6)+\delta(a_0)+...+\delta(a_k).$$
We know already that $\delta(6)\le 6$ and since each $a_i$ must be one of $0,1,2,3,4,$ or $5$ we can conclude $\delta(a_i)\le 7$. That gives us
$$\delta(n)\le 6k+(k+1)7=13k+7.$$

And since $n$ has $\lfloor \log_{6}n\rfloor+1=k+1$ digits in its base-$6$ representation, we know $k=\lfloor \log_{6}n\rfloor\le \log_{6}n.$ So finally by substitution
$$\delta(n)\le 13\log_6 n + 7$$
\end{proof}

This is the bound we've been looking for! It is a rope taking us within arm's reach of the summit.

\section{Powers of Five}
\begin{center}
    ``\textit{Wanna see me run to that mountain and back?...}" \\
    -Spongebob (as \textit{The Quickster})
\end{center}

We should start with

\begin{lemma}
The last decimal digit of every power of $5$ is $5$.
\end{lemma}
\begin{proof}
There are a few ways to do this. One method is to find a way of writing $5^n$ as $10k+5$ for some integer $k$. The following will do:
$$5^n=10\cdot\frac{5^{n-1}-1}{2}+5.$$
Note that since powers of five are always odd, we are assured that $5^{n-1}-1$ is even and therefore that $\frac{5^{n-1}-1}{2}$ is a whole number. 
\end{proof}

The lemma guarantees that $5^n$ always has a $5$ in its decimal representation. That makes selfcondensing it much easier since we can take that last $5$ and use it as the base in ``$5^n$". Then we are free to condense the remaining digits into $n$. 

\begin{lemma}
$5^n$ is selfcondensable for $n\ge 53$.
\end{lemma}
\begin{proof}
The number $5^n$ yields $\lfloor\log_{10}5^n\rfloor+1$ decimal digits (which we know by Lemma 4.2). One such digit will be a $5$. Therefore $5^n$ is selfcondensable if the remaining $\lfloor\log_{10}5^n\rfloor$ digits can be condensed into $n$. But by Theorem 6.2, we know that $n$ can always be condensed from at most $13\log_6n + 7$ digits. So $5^n$ is selfcondensable if 
$$13\log_6n + 7\le \lfloor\log_{10}5^n\rfloor.$$
The floor function (= rounding down) makes the inequality a bit troublesome so instead we choose to work with
$$n\log_{10}5 - 1\le \lfloor n\log_{10}5\rfloor=\lfloor\log_{10}5^n\rfloor.$$
So when is
$$13\log_6n + 7\le n\log_{10}5 - 1?$$
The inequality is not true for small integers. But it must flip-flop eventually since the left-hand-side grows logarithmically and the right-hand-side grows linearly. Some button-punching reveals the flip-flop at $n=53$ since
$$13\log_6(53) + 7=35.806...\quad\text{and}$$
$$53\log_{10}5 - 1=36.054...\ .$$
\end{proof}

This lemma is pretty good. But with a quick trick we can do better than $n\ge 53$.
\begin{lemma}
$5^n$ is selfcondensable for $n\ge 24$.
\end{lemma}
\begin{proof}
Just compare $\lfloor \log 5^n \rfloor$ to the $\delta(n)$ table we made in the previous section.
\end{proof}

We are ready now. We have approached the summit with toolery so far but the final ascent must be done free-climbing.

\begin{theorem}
$5^n$ is selfcondensable for all $n\ge 1$.
\end{theorem}
\begin{proof}

The previous two lemmas cover everything except $n=1,2,...,23$. We condense these last values with our bare hands:
\begin{center}
    \begin{tabular}{c|c|c}
        $n$ & $5^n$ & condensing of $n$ \\ \hline
        $2$ & $25$ & $\text{just}\ 2$ \\
        $3$ & $125$ & $1+2$ \\
        $4$ & $625$ & $6-2$ \\
        $5$ & $3125$ & $3\cdot1 + 2$ \\
        $6$ & $15625$ & $6\cdot1^{5+2}$ \\
        $7$ & $78125$ & $7\cdot 1^{8+2}$ \\
        $8$ & $390625$ & $6+2+0\cdot(3+9)$ \\
        $9$ & $1953...$ & $9\cdot1^{5+3+...}$ \\
        $10$ & $9765625$ & $9+7-6\cdot(6-5)^2$ \\
        $11$ & $48828...$ & $ 4+8-\big(\frac{8}{8}\big)^{2+...}$ \\
        $12$ & $2441...$ & $4\cdot(2+1^{4+...})$ \\        
        $13$ & $12207...$ & $7\cdot2 - 1^{2+0+...}$ \\
        $14$ & $61035...$ & $(6+1)(3-0!^{5+...}) $ \\
        $15$ & $3051...$ & $3\cdot5\cdot0!^{1+...}$ \\
        $16$ & $1525...$ & $2^{5-1^{5+...}}$ \\
        $17$ & $76293...$ & $2\cdot9 -(7-6)^{3+...}$ \\
        $18$ & $38146...$ & $3\cdot8 - 6\cdot 1^{4+...}$ \\
        $19$ & $190734...$ & $9+7+3\cdot1^{0+4+...}$ \\
        $20$ & $953674...$ & $ 9+5+6\cdot(4-3)^{7+...}$ \\
        $21$ & $476837...$ & $7\cdot3\cdot(7-6)^{4+8+...} $ \\
        $22$ & $23841...$ & $2\cdot(3+8)\cdot1^{4+...}$ \\
        $23$ & $1192092...$ & $(1+1)9+2+2+0!^{9+...}$ \\
    \end{tabular}
\end{center}

Lastly, the case $n=1$ is so simple it is hardly worth pointing out. To say that $5$ is selfcondensable is really to say that $5$ is $5$.
\end{proof}

\section{Fewer Rules}
\begin{center}
    ``\textit{... Wanna see me do it again?}" \\
    -Spongebob (as \textit{The Quickster})\\\ \\
    ``\textit{Hey Goku -- why don't you try using some heavier weights?}" \\
    -King Kai, \textit{Dragon Ball Z}
\end{center}

Once a difficult goal has been accomplished it is natural -- and often very fun -- to attempt the same goal with more difficulty by one's own choice. This is done by adding restrictions such as, for example, in competitive Rubiks cube solving. After mastering puzzle solving with hands some solvers restrict themselves to use only feet\footnote{Which has been done in under 17 seconds: \texttt{https://www.youtube.com/watch?v=IZjToD9B4Vw}}. Similarly, many humans have managed to be shot out of cannons. But this is not difficult enough. Some humans have chosen to be shot out of cannons on fire\footnote{See \textit{America's Got Talent}: \texttt{https://www.youtube.com/watch?v=G4kyaGML2pM}}. And again, the first recorded summiting of Mt. Everest was 1953 using supplemental oxygen. Then in 1978, an expedition summited restricting themselves to only natural oxygen.

But do not think these extremeties are natural only to a weird minority. Children know better. No sooner has a child cannonballed into the pool then they are jumping in again -- this time spinning, next with eyes closed, next from the taller diving board, and so on. The author is not sure what is the best articulation of such activity. One sometimes hears it called ``pushing the limits" but that is not quite right. Very often, we are searching out and putting on limits like shoppers trying new outfits in the dressing room. Rather, we are ``pushing" into chaos using limitation as a hatchet cutting out a path.

Every art -- music, sports, mathematics, philosophy, animation, etc. -- is therefore of a battle against chaos. Every discover and record is a push one inch further into the unknown. And we all know this already. It is contact with the boundary of chaos that causes us to rise to our feet cheering at say football games, rodeos, Olympic gymnastics, or competitive Smash Bros. The deepest appetite of the human soul is to conquer chaos and each of us \textit{does} conquer -- until we grow apathetic and forget how to be childish.

But one should check motive. There is always the temptation to push boundaries to show off or simply to irritate others. We must do our work for the high and hard duty of -- our God help us -- having fun.

We reached our summit in the previous section but did so with supplemental oxygen. We used 6 rules as we climbed and we now challenge ourselves to do without all six. And so we step into chaos. Some rules can be scrapped easily. For example, the most important use we made of division was turning equal digits into ones (ex. $\frac{7}{7}=1$). But there is a workaround with subtraction and factorial since $(7-7)!=0!=1$. So having thrown division overboard, some $\delta(n)$ values might become bigger -- but not terribly. At worst, we would need only do more cases by hand.

But can we scrap factorial? Yes. The really important use of factorial was turning zeros into ones. But we can workaround this with exponents since mathematicians are generally agreed $0^0=1$. Once again, this will increase $\delta(n)$ values a bit. Most $\delta(n)$ values will double but the worst ones will perhaps only triple or quadruple. This is fine since Our proof relied on a linear function outrunning a logarithm. And quadrupling $\delta(n)$ will only slow the line down -- and even the slowest line can outrun a logarithm. Scrapping rules is easy so far.

What about exponents? Can we scrap them? No -- nothing is selfcondensable (well, except single digit numbers) if we do so. Our whole proof depended on using the exponent-nature of $5^n$. In general, whole numbers become too big and multiplication cannot catch up fast enough. For example, selfcondensing $1381$ is impossible because the biggest target value we can manage is a measly $(1+1)\cdot3\cdot8=48.$

There are two interesting sets of rules we should talk about though.

1) Imagine we can use the digits as many times as we want but can only use addition, subtraction, and multiplication a limited number of times. Here we have stumbled into the foothills of one of the harshest mountains in the mathematical landscape\footnote{Thanks to Sam Spiro pointing this out.}. Suppose we are handed an arbitrary whole number and asked how many operations are needed to condense it from its digits. It turns out a formula or quick method to do this might give us a solution to the ``P=NP" problem\footnote{Technically, we are supposed to start with only the number 1 and use operations from there (as opposed to starting with the number's own digits). Also we only need a good formula for $n!$ and not arbitrary whole numbers. Mostly we should note that we have entered a territory which only the most skilled and resilient of our mathematical-alpiners have started to navigate.} -- that notoriously difficult peak which seems to have no climbable faces. We are not climbing that mountain (not today at least).

2) We keep factorial and exponents in general but cannot use them to turn zeros into ones. In other words, the equations ``$0!=1$" and ``$0^0=1$" are not allowed. Selfcondensability in general will be difficult since big numbers with lots of zeros like $120003000000000000000004$ are almost certainly not selfcondensable. We therefore must show that there are not too many zeros in any power of $5$. The first zero, for example, occurs at $5^8=390625$ and the power with the largest percentage of zeros seems to be
$$5^{45}=284217\mathbf0943\mathbf0404\mathbf0\mathbf0743484497\mathbf07\mathbf03125$$
in which about $ 22$\% of the digits are zero (every power up to $5^{1000}$ was checked). We should give this case its own section.

\section{A Dead End}

We want to explore a particular restriction of our original six rules ans see if the powers of five can still be proven selfcondensable. We have scrapped some of our gear and want to know if the mountain can still be climbed. The restriction is zeros can no longer be made ones (either by factorial, $0!=1$, or by exponents, $0^0=1$). Zeros must remain boring old zeros. The proof comes down to whether or not powers of five have too many zeros.

The goal seems obvious. If all digits are equally likely, shouldn't roughly $10$\% be zero? And shouldn't the other $90$\% of non-zero digits be enough to do our work? This appears the case for big powers at least. For example, $5^{1430}$ has $1000$ digits with the following count of each:
\begin{center}
\begin{tabular}{ c | c c c c c c c c c c}
 digit & 0 & 1 & 2 & 3 & 4 & 5 & 6 & 7 & 8 & 9\\ \hline
 \# in $5^{1430}$ & 98 & 97 & 89 & 92 & 94 & 104 & 114 & 91 & 115 & 106
\end{tabular}
\end{center}

The percentage of zeros in $5^n$ likely gets closer to $10$\% as $n$ gets bigger. But \textit{likely} is not synonomous with \textit{proven} for mathematicians anymore than looking at a picture of a summit is the same as standing on it for a mountaineer. We must get our hands dirty before deciding the matter.

The author however failed to find a proof. The peak is very difficult with our restricted gear. We are forced to analyze the decimal representations of integer powers -- something on which very little is known. The task is like attempting surgery on a creature no one knows the anatomy of. We are forced on a detour through an uncharted mountain range. One of the great mathematical alpiners, Paul Erd\H{o}s, made some guesses about the digits of powers of $2$. We still don't know if they are correct though many mathematicians have tried hard to find out -- some with very sophisticated climbing gear\footnote{See \textit{Ternary expansions of Powers of 2} at \texttt{https://arxiv.org/abs/math/0512006}}. But even if a summit cannot yet be reached, mathematicians still have fun guessing what the wildlife will look like once up top.

We will hike a route which appeared promising but dead ends. There is, at least, good scenery on the way. Hopefully a better climber will complete the route or discover a new one.

To find our trailhead, we play a game. How many zeros in a row can one find in powers of five? -- any success? The first zero is in $5^8=39\mathbf0625$, the first pair in 
$$5^{39}=181898940354585647583\mathbf{00}78125,$$
the first $3$-in-a-row in
$$5^{67}=677626357803440271254658\mathbf{000}54371356964111328125$$
and the first $4$-in-a-row in $5^{228}$. Can we find as many zeros in a row as we want? -- or is there a limit? If a limit, we can reach the summit. Imagine, for example, no more than $4$ zeros in a row ever appear. Then roughly $20$\% non-zero digits is guaranteed since in any $5$ consecutive digits, at most $4$ are zero. $20$\% is, of course, much lower than our original $90$\% estimate. But it is enough for a proof. Really any threshold besides $0$\% is enough.

But, as it turns out, no limit exists. One can find any amount of zeros in a row if only one looks far enough. And this we can prove. It will be messy though -- but the best things usually are (such as chili or the birth of puppies). If we must get our hands dirty we may as well jump in the mud.

\section{Perhaps Jargon}
\begin{center}
    ``\textit{Jargon, not argument, is your best ally}" \\
    -C.S. Lewis, \textit{Screwtape Letters}
\end{center}

Most embarrassingly, the author could not present this section without highly technical language -- which is another way to say the author does not understand very well what they are saying. The reader's forgiveness is therefore requested if this section turns out only a demonstration of jargon.

We must make up words again. ``$S_k$" will mean ``the $k$th digits, counting from the right, in powers of five". So for example, $S_3=(00003580358...)$ since

\begin{center}
    \begin{tabular}{cr}
        $5^1=$ & $\mathbf0005$ \\
        $5^2=$ & $\mathbf0025$ \\
        $5^3=$ & $\mathbf0125$ \\
        $5^4=$ & $\mathbf0625$ \\
        $5^5=$ & $\mathbf3125$ \\
        $5^6=$ & $1\mathbf5625$ \\
        $5^7=$ & $7\mathbf8125$ \\
        $5^8=$ & $39\mathbf0625$ \\
        $5^9=$ & $195\mathbf3125$ \\
        $5^{10}=$ & $976\mathbf5625$ \\
        $5^{11}=$ & $4882\mathbf8125$ \\
        $5^{12}=$ & $...$ \\
    \end{tabular}
\end{center}
Similarly,
$$S_0=(555...),\quad S_1=(0222...),\quad S_2=(001616...),\quad\text{and}\quad S_4=(00000179562...).$$
Some (or all?) of the $S_k$ enter cycles. When this happens we use ``$C_k$" to mean ``just the cycle of $S_k$". So 
$$C_0=(5),\quad C_1=(2),\quad C_2=(16),\quad \text{and}\quad C_3=(0358).$$

Let's also use ``$d_n(k)$" to mean ``the $k$th digit of $5^n$ counting from the right". So we can rephrase the definition of $S_k$ in Mathlish as ``$S_k=\{d_n(k)\}_{n=0}^\infty$". In fact, the definition of $d_n(k)$ can be rephrased as ``$5^n=\sum_{k=0}^\infty d_n(k)10^k$" -- well, partially rephrased. To be complete we need also to communicate $d_n(k)$ is a single decimal digit by saying ``$0\le d_n(k)\le 9$ for all $n, k$".

Lastly, we need a word to communicate whether a whole number is odd or even (called the number's \textit{parity} usually). We use ``$P$" and say $P(x)=0$ when $x$ is even and $P(x)=1$ when $x$ is odd.

With all that we are ready for the trail. We have really only been tying our shoelaces up until now. The first switchback is

\begin{lemma}
    The $d_n(k)$ all obey
    $$d_{n+1}(k)=5P(d_n(k))+\Big\lfloor\frac{d_n(k-1)}{2}\Big\rfloor.$$
\end{lemma}
\begin{proof}
    This becomes obvious after an hour of pencil-and-paper multiplication, but we should prove it anyways. It is sad no better presentation could be quickly found besides this ugly equation.
    
    We can rewrite a whole number in a way that breaks off its last digit as one breaks a square off a chocolate bar. For example, $945=5+94\cdot10$ or $1729=9+172\cdot10$. In general, given an integer $x$, we can rewrite it
    $$x=(x\ \%\ 10)+\Big\lfloor\frac{x}{10}\Big\rfloor\cdot 10.$$
    
    This observation is all we need. The rest can be said in algebra:
    $$\sum_{k=0}^\infty d_{n+1}(k)10^k$$
    $$=5^{n+1}=5\cdot 5^n=\sum_{k=0}^\infty 5d_{n}(k)10^k$$
    $$=\sum_{k=0}^\infty \Big((5d_{n}(k)\ \%\ 10)+\Big\lfloor\frac{5d_{n}(k)}{10}\Big\rfloor\cdot 10\Big)10^k$$
    $$=\sum_{k=0}^\infty 5P(d_n(k))10^k+\Big\lfloor\frac{d_{n}(k)}{2}\Big\rfloor\cdot 10^{k+1}$$
    $$=\sum_{k=0}^\infty \Big(5P(d_n(k))+\Big\lfloor\frac{d_{n}(k-1)}{2}\Big\rfloor\cdot \Big)10^{k}$$
    
    If the substitution of ``$5P(d_n(k))$" for ``$5d_n(k)\ \%\ 10$" is confusing, note any integer $x$ can be written $x=2m+P(x)$ for some integer $m$. Thus
    $$5x\ \%\ 10=5(2m+P(x))\ \%\ 10=(10m+5P(x))\ \%\ 10=5P(x).$$
\end{proof}

This lemma is like a biologist's scalpel. We are ready to dissect $S_k$.

\begin{theorem}
    For $k\ge2$, every $S_k$ enters a cycle \\
    \indent (a) by $n=k+1$\\
    \indent (b) of length $l=2^{k-1}$ \\
    \indent (c) in which $d_{n+l/2}(k)=d_n(k)\pm 5$ and\\
    \indent (d) if the occurence count of each digit is recorded in $\vec{v_k}\in\mathbb{Z}^{10}$ (ex. the $4$th component of $\vec{v_k}$ is the total amount of $4$'s in the cycle of $S_k$) then $\vec{v_k}=A^{k-1}\vec{u}$ where
    $$
    A=
    \begin{bmatrix}
    1 & 1 & 0 & 0 & 0 & 0 & 0 & 0 & 0 & 0 \\
    0 & 0 & 1 & 1 & 0 & 0 & 0 & 0 & 0 & 0 \\
    0 & 0 & 0 & 0 & 1 & 1 & 0 & 0 & 0 & 0 \\
    0 & 0 & 0 & 0 & 0 & 0 & 1 & 1 & 0 & 0 \\
    0 & 0 & 0 & 0 & 0 & 0 & 0 & 0 & 1 & 1 \\
    1 & 1 & 0 & 0 & 0 & 0 & 0 & 0 & 0 & 0 \\
    0 & 0 & 1 & 1 & 0 & 0 & 0 & 0 & 0 & 0 \\
    0 & 0 & 0 & 0 & 1 & 1 & 0 & 0 & 0 & 0 \\
    0 & 0 & 0 & 0 & 0 & 0 & 1 & 1 & 0 & 0 \\
    0 & 0 & 0 & 0 & 0 & 0 & 0 & 0 & 1 & 1 \\
    \end{bmatrix}
    \quad\text{and}\quad
    \vec{u}=
    \begin{bmatrix}
    0 \\ 0 \\ 1 \\ 0 \\ 0 \\ 0 \\ 0 \\ 0 \\ 0 \\ 0
    \end{bmatrix}
    $$.
\end{theorem}
\begin{proof}
    We induct $k$ starting with $k=2$. 
    
    Note $d_n(0)=5$ and $d_n(1)=2$ for $n\ge 2$ (this can be proven rigorously inducting $n$ with Lemma 10.1). Next inspect $d_n(2)$ and observe $d_3(2)=1$ since $5^3=125$. Thus $d_n(2)$ therefore oscillates between $1$ and $6$ since
    $$6=5P(1)+\Big\lfloor\frac{d_n(1)}{2}\Big\rfloor=5+1$$
    $$\text{and}\quad1=5P(6)+\Big\lfloor\frac{d_n(1)}{2}\Big\rfloor=0+1.$$
    Thus we have a cycle (a) starting at $n=k+1=3$ (b) of length $l=2^{k-1}=2$ (c) obeying $d_{n+1}(2)=d_n(2)\pm 5$ and (d) satisfying 
    $$\vec{v_2}=A\vec{u}=
    \begin{bmatrix}
    0 & 1 & 0 & 0 & 0 & 0 & 1 & 0 & 0 & 0
    \end{bmatrix}^\top.$$
    
    The inductive case is trickier. We must prove the theorem for $k$ assuming it for $k-1$.
    
    We should first simply convince ourselves $S_k$ cycles. by Lemma 10.1, $S_k$ cycles if any pair $(d_n(k), d_n(k-1))$ occurs at two values of $n$. But the pair can only be one of $10^2=100$ possible values and therefore recurs. And since $d_n(k)$ affects $d_{n+1}(k)$ only by its parity, the length of $C_k$ must be either equal to or double the length of $C_{k-1}$.
    
    Next we inspect where $C_k$ begins. Since it is assumed $C_{k-1}$ starts by $n=k$ with length $2^{k-2}$ we are assured
    $$d_{k-1}(k-1)=d_{k-1+2^{k-2}}(k-1).$$
    And note $d_k(k)=0$ since $10^k>5^k$ for $k\ge 1$. Thus if $d_{k+2^{k-2}}(k)$ is even, $C_k$ begins by $n=k+1$ and is equal in length to $C_{k-1}$ (so $l=2^{k-2}$). If instead $d_{k+2^{k-2}}(k)$ is odd, we claim $d_{k+2^{k-1}}(k)$ is even and $C_k$ therefore still starts by $n=k+1$ but is double the length of $C_{k-1}$ (so $l=2^{k-1}$).
    
    If follows from Lemma 10.1
    $$P(d_{n+1}(k)-d_n(k))=P\Big(\Big\lfloor\frac{d_n{k-1}}{2}\Big\rfloor\Big).$$
    Thus $d_{n+1}(k)$ differs in parity from $d_n(k)$ exactly when $d_n(k-1)$ is one of $2,3,6,$ or $7$. $d_{k+2^{k-2}}(k)$ is therefore odd only if $C_{k-1}$ contains an odd count of $2$'s, $3$'s, $6$'s, and $7$'s. In which case, those same digits ensure $d_{k+2^{k-1}}(k)$ is even.
    
    We must therefore prove $C_{k-1}$ contains an odd count of $2$'s, $3$'s, $6$'s, and $7$'s for all $k\ge 2$. At this point, part (d) of our assumption is needed. We analyze the orbit of $\vec{u}$ under $A$ in $\mathbb{F}_2^{10}$ and total up the $2$nd $3$rd, $6$th, and $7$th components:
    \begin{center}
    \begin{tabular}{ c | c c c c c c c}
    comp. & $\vec{u}$ & $\vec{v_2}$ & $\vec{v_3}$ & $\vec{v_4}$ & $\vec{v_5}$ & $\vec{v_6}$ & $\vec{v_7}$ \\ \hline
    0 & 0 & 0 & 1 & 1 & 0 & 1 & 1 \\
    1 & 0 & 1 & 0 & 1 & 1 & 0 & 0 \\
    2 & \textbf1 & \textbf0 & \textbf0 & \textbf1 & \textbf0 & \textbf1 & \textbf0 \\
    3 & \textbf0 & \textbf0 & \textbf1 & \textbf0 & \textbf0 & \textbf1 & \textbf1 \\
    4 & 0 & 0 & 0 & 1 & 1 & 1 & 0 \\
    5 & 0 & 0 & 1 & 1 & 0 & 1 & 1 \\
    6 & \textbf0 & \textbf1 & \textbf0 & \textbf1 & \textbf1 & \textbf0 & \textbf0 \\
    7 & \textbf0 & \textbf0 & \textbf0 & \textbf1 & \textbf0 & \textbf1 & \textbf0 \\
    8 & 0 & 0 & 1 & 0 & 0 & 1 & 1 \\
    9 & 0 & 0 & 0 & 1 & 1 & 1 & 0 \\ \\
    $\sum$ & \textbf1 & \textbf1 & \textbf1 & \textbf3 & \textbf1 & \textbf3 & \textbf1
    \end{tabular}
    \end{center}
    
    The residues of $\vec{v_3}$ and $\vec{v_7}$ are equal. Therefore $\vec{v_k}\equiv\vec{v_{k+4}}\mod 2$ for all $k\ge 3$ and $\vec{v_{k-1}}$ in our present case must drop to one of the $6$ former residues, each of which has an odd total of $2$'s, $3$'s, $6$'s, and $7$'s.
    
    It follows $d_{k+2^{k-2}}(k)$ is odd, $d_{k+2^{k-1}}(k)$ is even, and $C_k$ starts by $n=k+1$ with length $l=2^{k-1}$. Parts (a) and (b) are proven. 
    
    We go on to (c) and (d) which both describe the content of $C_k$. We must determine precisely which patterns $C_k$ inherits from $C_{k-1}$. Let's start by describing how exactly $d_n(k-1)$ affects $d_{n+1}(k).$ We can create a useful table from Lemma 10.1:
    \begin{center}
    \begin{tabular}{c|c|c}
    & $d_{n+1}(k)$ when & $d_{n+1}(k)$ when \\
    $d_n(k-1)$ & $d_n(k)$ is even & $d_n(k)$ is odd \\ \hline
    0 & 0 & 5 \\
    1 & 0 & 5 \\
    2 & 1 & 6 \\
    3 & 1 & 6 \\
    4 & 2 & 7 \\
    5 & 2 & 7 \\
    6 & 3 & 8 \\
    7 & 3 & 8 \\
    8 & 4 & 9 \\
    9 & 4 & 9 \\
    \end{tabular}
    \end{center}
    
    This is actually enough information to determine precisely the contributions each digit of $C_{k-1}$ makes to $C_k$. Since $C_k$ is double the length of $C_{k-1}$, every entry of $C_{k-1}$ makes exactly two contributions to $C_k$. Further, each of the only two possible contributions (as described in the former table) are each made once. This is because the contribution depends on the parity of $d_n(k)$. And as was noted while proving parts (a) and (b), $d_k(k)$ and $d_{k+2^{k-2}}(k)$ have different parity. That is enough for part (c). We now must wrap up part (d).
    
    Every entry of $C_k$ makes two distinct contributions given in the former table. The entries of $\vec{v_k}$ are therefore linear in those of $\vec{v_{k-1}}$. The matrix $A$ given in the theorem statement can be seen to follow from the table. The vector $\vec{u}$ is chosen simply for the correct initial condition and can therefore be interpreted as a consequence of $C_1=(2)$.
\end{proof}
    
The theorem gives us more information than needed -- meaning there are potential trails to the summit. But we were looking for zeros in a row. To get back on track we need

\begin{lemma}
    $C_k$ begins with at least $\lfloor k\log_k 10 \rfloor -k$ zeros.
\end{lemma}
\begin{proof}
    This is really a fancy way of saying $10^k$ is bigger than $5^k$. If $d_n(k)=0$ then $5^n<10^k$. Or equivalently, then $n<k\log_5 10$. But since $n$ is an integer we may write
    $$n\le \lfloor k\log_5 10\rfloor .$$
    The theorem tells us $C_k$ begins by $n=k+1$. The cycle must therefore start with 
    $$\lfloor k\log_5 10\rfloor-(k+1)+1=\lfloor k\log_5 10\rfloor-k$$
    zeros.
\end{proof}

This finally gives us
\begin{theorem}
    The amount of zeros in a row in $5^{m+2^m+2}$ is non-decreasing and grows to infinity as $m$ gets bigger.
\end{theorem}
\begin{proof}
    $C_k$, by the theorem, cycles with $l=2^{k-1}$ by $n=k+1$. Thus $5^{k+1}$ and $5^{k+1+2^{k-1}}$ have the same last $k+1$ digits. $5^{k+1}$ will contain a zero in the $r$th place if $C_{k-r}$ begins with at least $r+1$ zeros -- which, by the previous lemma, we are guaranteed to happen for any $r$ by simply making $k$ big enough. Substituting $m=k-1$ gives the result (it looks a bit cleaner with $m$).
\end{proof}

We should look at our results in the broad daylight before concluding.
$$5^{8}= 39\mathbf{0}625$$
$$5^{13}= 12207\mathbf{0}3125$$
$$5^{22}= 2384185791\mathbf{0}15625$$
$$5^{39}= 181898940354585647583\mathbf{00}78125$$
$$5^{72}= 2117582368135750847670806251699104905128479\mathbf{00}390625$$
$$5^{137}= ...27\mathbf{00}1953125$$
$$5^{266}= ...51\mathbf{000}9765625$$
$$5^{523}= ...63\mathbf{000}48828125$$
$$5^{1036}= ...19\mathbf{000}244140625$$
$$5^{2061}= ...47\mathbf{000}1220703125$$
$$5^{4110}= ...11\mathbf{0000}6103515625$$
$$5^{8207}= ...43\mathbf{0000}30517578125$$
$$5^{16400}= ...59\mathbf{0000}152587890625$$
$$5^{32785}= ...67\mathbf{00000}762939453125$$
$$5^{65554}= ...71\mathbf{00000}3814697265625$$
$$5^{131091}= ...23\mathbf{00000}19073486328125$$
$$5^{262164}= ...99\mathbf{000000}95367431640625$$

\section{Philisophical Suspicions}
\begin{center}
    ``\textit{But, dear father, upon what grounds are you so opposed to belief in dreams}" \\
    -Novalis, \textit{Henry of Ofterdingen}
\end{center}

Both our goals have been reached (well -- technically we reached one goal and found an interesting dead end to another). The powers of five have been proven selfcondensable with our original rules. Any amount of zeros in a row can be guaranteed to appear in the powers of five at some point. Give yourself a pat on the back.

The author wanted at first to fill this section with philosophy -- with clever ideas slightly relevant to the preceding mathematics. They wanted to explain how true freedom is not a lack of restriction but is actually the selection of restriction. They would have said the best freedom consists of the funnest -- and often strictest -- rules.
    
The author wanted to point out how mystic language and jargon are indistinguishable at first glance -- how prophets, mathematicians, poets, and theologians often get ideas into their heads so big spoken language is too small to fit them. They would have mentioned the strange affinity of mathematics for religious experience and insanity. The author would have mentioned Ramanujan, Grothendieck, Cantor, or Nash as examples. They would have quoted G.K. Chesterton to you:
\begin{center}
    ``\textit{The poet only asks to get his head into the heavens. It is the logician who seeks to get the heavens into his head. And it is his head that splits.}" \\
    -\textit{Orthodoxy}
\end{center}

They would have mentioned the theologian Thomas Aquinas who said of his work ``\textit{it reminds me of straw}". They would have mentioned how great religious minds grappling with conceptual structures too large for their language are pushed back either into (1) paradox like Ezekiel:
\begin{center}
    ``\textit{As I looked at the living creatures, I saw a wheel on the ground beside each creature with its four faces. This was the appearance and structure of the wheels: They sparkled like topaz, and all four looked alike. Each appeared to be made like a wheel intersecting a wheel.}" \\
    -\textit{Ezekiel 1:15-16}
\end{center}
or (2) mysticism like John:
\begin{center}
    ``\textit{ I will also give that person a white stone with a new name written on it, known only to the one who receives it.}" \\
    -\textit{Revelation 2:17}
\end{center}
or (3) simply admit inability like Dante\footnote{The author apologizes that only examples from the Christian religion are provided. The author is largely ignorant of other traditions and is happy to receive communication from any reader who knows examples beyond the territory of Christendom.}:
\begin{center}
    ``\textit{ Had I a tongue in eloquence as rich,\\
As is the colouring in fancy's loom,\\
'T were all too poor to utter the least part\\
Of that enchantment.}" \\
    -\textit{Paradiso XXXI}
\end{center}

The author would have drawn out from the similarity between mystic language and jargon interesting conclusions about how humans think. But they found these philosophical suspicions too difficult to explain and therefore leaves them as trailheads to the real hikers.


\begin{thebibliography}{0}
\bibitem{moscow} 
Boris A. Kordemsky (1992). \textit{The Moscow Puzzles - 359 Mathematical} \\ \textit{Recreations}, Dover Publications

\end{thebibliography}
\end{document}